\documentclass[12pt]{article}
\usepackage{geometry}
\usepackage[ansinew]{inputenc}
\usepackage{comment}
\usepackage{amssymb}
\usepackage{latexsym}
\usepackage{graphics}
\usepackage{epstopdf}
\usepackage{color}
\usepackage{fancybox}
\usepackage{amsmath,amsfonts}

\begin{document}
\setlength{\parskip}{0.3\baselineskip}

\newtheorem{theorem}{Theorem}
\newtheorem{corollary}[theorem]{Corollary}
\newtheorem{lemma}[theorem]{Lemma}
\newtheorem{proposition}[theorem]{Proposition}
\newtheorem{definition}[theorem]{Definition}
\newtheorem{remark}[theorem]{Remark}
\renewcommand{\thefootnote}{\alph{footnote}}
\newenvironment{proof}{\smallskip \noindent{\bf Proof}: }{\hfill $\Box$\hspace{1in} \medskip \\ }


\newcommand{\beqaa}{\begin{eqnarray}}
\newcommand{\eeqaa}{\end{eqnarray}}
\newcommand{\beqae}{\begin{eqnarray*}}
\newcommand{\eeqae}{\end{eqnarray*}}


\newcommand{\sii}{\Leftrightarrow}
\newcommand{\imer}{\hookrightarrow}
\newcommand{\imerc}{\stackrel{c}{\hookrightarrow}}
\newcommand{\Con}{\longrightarrow}
\newcommand{\con}{\rightarrow}
\newcommand{\conf}{\rightharpoonup}
\newcommand{\confe}{\stackrel{*}{\rightharpoonup}}
\newcommand{\pbrack}[1]{\left( {#1} \right)}
\newcommand{\sbrack}[1]{\left[ {#1} \right]}
\newcommand{\key}[1]{\left\{ {#1} \right\}}
\newcommand{\dual}[2]{\langle{#1},{#2}\rangle}
\newcommand{\intO}[1]{\int_{\Omega}{#1}\, dx}


\newcommand{\R}{{\mathbb R}}
\newcommand{\N}{{\mathbb N}}
\newcommand{\Z}{{\mathbb Z}}
\newcommand{\Q}{{\mathbb Q}}

\newcommand{\cred}[1]{\textcolor{red}{#1}}

\title{\bf A Note on the Regularity of Thermoelastic Plates with Fractional Rotational Inertial Force  }
\author{
Fredy Maglorio Sobrado  Su\'arez  \\
{\small Department of Mathematics,  Federal University of Technological  of Paran\'a, Brazil}
}
\date{}
\maketitle

\let\thefootnote\relax\footnote{{\it Email address:}   {\rm fredy@utfpr.edu.br} (Fredy Maglorio Sobrado  Su\'arez)}.


\begin{abstract}
The present work intends to complement the study of the regularity of the solutions of the thermoelastic plate with rotacional forces. The rotational forces involve the spectral fractional Laplacian, with power parameter $\tau\in [0,1]$ ( $\gamma(-\Delta) ^\tau u_{tt}$).  Previous research regarding regularity showed that,      as for the analyticity of the semigroup $S(t)=e^{\mathbb{B}t}$  for the Euler-Bernoulli Plate($\tau=0$) model,  the first result was established by Liu and Renardy, \cite{LiuR95} in the case of hinged and clamped boundary conditions, for the case $\tau=1$ (Plate Kirchoff-Love) Lasiecka and Triggiani showed,   that the semigroup is not differentiable \cite{LT1998, LT2000} and more recently in 2020 Tebou et al.\cite{Tebou2020} showed that for $\tau\in (0,\frac{1}{2})$, $S(t)$ is of class Gevrey $s>\frac{2-\tau}{2-4\tau}$.  Our main contribution here is to show that $S(t)$ is of Gevrey class $s>\frac{3-\tau}{2-2\tau}$ when the parameter  $\tau$ lies in the interval $[\frac{1}{2},1)$ and also show that $S(t)$ is not analytic for $\tau\in (0,1]$   both results for Hinged plate/ Dirichlet temperature  boundary conditions.
\end{abstract}

\bigskip
{\sc keyword:} Euler-Bernoulli Plate,  Plate Kirchoff-Love,  Gevrey Class,  Fractional Rotational Inertial force,  Analiticity, Thermoelastic Plates.

\setcounter{equation}{0}

\section{Introduction}
Consider that $\Omega$ is a bounded open subset of $\mathbb{R}^n$, $n\geq 1$,  with sufficiently smooth boundary. The system  we study is given by the following  coupled plate equations:
\begin{eqnarray}
\label{ISplacas0II}
    u_{tt}+\gamma (-\Delta)^\tau u_{tt}+\Delta^2u+\alpha\Delta\theta=0,\quad
    &x\in\Omega,&
t>0, \\
\label{ISplacas1II}
    \theta_t-\kappa\Delta\theta-\beta\Delta  u_t =0,\quad
    &x\in\Omega,&
     t>0,
\end{eqnarray}
satisfying the  boundary conditions
\begin{equation}\label{ISplacas4II}
    u=\Delta u=0,\quad \theta=0,\quad x\in\partial\Omega,\ t>0,
\end{equation}
and prescribed initial data
\begin{eqnarray}
\label{ISplacas2II} u(x,0)=u_0(x),\ u_t(x,0)=u_1(x),\
\theta(x,0)=\theta_0(x), \quad x\in\Omega.
\end{eqnarray}
Here $u$, denote the transversal displacements of the plates. The  coefficients and the inertia rotational is  given by
  $\gamma$  positive number. The exponent $\tau$ lies in the interval $[0,1]$, the positive numbers $\alpha$  and $\beta$ are  the coupling coefficient.

The system we study is given by the following  coupled plate equations:
\begin{eqnarray}
\label{ISplacas0}
    u_{tt}+\gamma (-\Delta)^\tau u_{tt}+\Delta^2u+\alpha\Delta\theta=0,\quad
    &x\in\Omega,&
t>0, \\
\label{ISplacas1}
    \theta_t-\kappa\Delta\theta-\beta\Delta  u_t =0,\quad
    &x\in\Omega,&
     t>0,
\end{eqnarray}

and prescribed initial data
\begin{eqnarray}
\label{ISplacas2} u(x,0)=u_0(x),\ u_t(x,0)=u_1(x),\
\theta(x,0)=\theta_0(x), \quad x\in\Omega.
\end{eqnarray}
Here $u$, denote the transversal displacements of the plates.  The  inertia rotational coefficients given by
 $\gamma$  positive number. The exponent $\tau$ lies in the interval   $[0,1]$,  $\alpha$  and $\beta $ constant positive are   the coupling coefficient.

Various researchers year after year have been devoting their attention to the study of the asymptotic behavior and regularity of the solutions of the thermoelastic plate system, especially when considering the rotational inertial force in the system given by: $\gamma(-\Delta)^\tau u_{tt}$.  We emphasize that this mathematical model, when the parameter $\tau$ takes the values $0$ and $1$, are called Euler-Bernoulli and Kirchoff-Love thermoelastic plates, respectively.  We know that regarding the asymptotic behavior the  best decay rate is in exponential rate,  the first results for the Euler-Bernoulli model, it is well known that the underlying semigroup is both analytic and exponentially stable in the case of hinged boundary conditions and clamped boundary condition.  However, for the Kirchhoff model, only the exponential stability of the semigroup is true; Lasiecka and Triggiani showed that the semigroup is not only not analytic, it is not even differentiable \cite{LT1998,  LT2000} for any of the hinged or clamped boundary conditions.

 Later, many other works followed, establishing the exponential stability of thermoelastic plates (Euler-Bernoulli and Kirchoff-Love) with various boundary conditions, e.g. \cite{AvalosL1997, AvalosL1998, LiuZ1997, LTebou2013}.  Regarding the analyticity of the semigroup for the Euler-Bernoulli model, the first result was established by Liu and Renardy, [19] in the case of bounded and articulated boundary conditions. Subsequently, Liu and Liu, \cite{LiuLiu1997}, and Lasiecka and Triggiani \cite{LT1998,LT1998A,LT1998B,  LT1998C} demonstrated other analyticity results under various boundary conditions.

In more recent research from 2020 Tebou et al. \cite{Tebou2020} studied thermoelastic plates considering the fractional rotational inertial force ($\gamma(-\Delta)^\tau u_{tt}$ for the parameter $\tau\in [0,1]$. In $\Omega$, limited open subset of $\mathbb{R}^n$, $n\geq 1$,  with smooth enough boundary In this research the authors prove that the semigroup associated to the system is the Gevrey class $s$ for each $s>\frac{2-\tau}{2-4\tau}$ for both:  the Hinged plate/Dirichlet temperature boundary conditions  and Clamped plate/Dirichlet temperature  boundary conditions when the parameter $\tau$ lies in the interval $(0,\frac{1 }{2})$, also show that the semigroup $S(t)$ is exponentially stable for Hinged boundary conditions,  for $\tau$ in the interval $ [0, 1]$ and finish their investigation,  constructing a counterexample, that,  under hinged boundary conditions,  the semigroup is not analytic, for all $ \tau$ in the interval $(0,1)$.  To determine the Gevrey class of $S(t)$ use the domain method of the frequency,  the appropriate decompositions of the components of the system and the use of Lions'  interpolation inequalities.  More recent research in this direction can be found at \cite{HPFredy2019,  HSLiuRacke2019, BrunaJMR2022, Tebou2021}.

The rest of this article is organized as follows: in section 2, we study the well-posedness of the system \eqref{ISplacas-20}-\eqref{ISplacas-35} through semigroup theory.  We leave our main contributions for the third section, which is subdivided into two subsections. In 3.1 we showed that the semigroup $S(t)=e^{\mathbb{B}t}$ is not analytic when $t\in(0 ,1]$ and in the last subsection 3.2 we showed that the underlying semigroup is Gevrey class $s$ for each $s>\frac{3-\tau}{2-2\tau}$ for Hinged plate/Dirichlet temperature boundary condition when the parameter $\tau$ is in the interval $[\frac{1}{2}, 1]$.  We end this investigation with one  observation of  the exponential decay of $S(t)$ for $\tau\in [0,1]$.

\subsubsection{Well-Posedness of the System of plates thermoelastic}
In this section we will use the semigroup theory for assure the existence and uniqueness of strong solutions for the system \eqref{ISplacas-20}-\eqref{ISplacas-35}.  Before this, we are going to recall some preliminary results.
\begin{theorem}[See Theorem 1.2.4 in \cite{LiuZ}] \label{TLiuZ}
Let $\mathbb{B}$ be a linear operator with domain $D(\mathbb{B})$ dense in a Hilbert space $\mathbb{X}$. If $\mathbb{B}$ is dissipative and $0\in\rho(\mathbb{B})$, the resolvent set of $\mathbb{B}$, then $\mathbb{B}$ is the generator of a $C_0$-semigroup of contractions on $\mathbb{X}$.
\end{theorem}
To re-write the system \eqref{ISplacas-20}-\eqref{ISplacas-35} in an abstract form, considering the operator self-adjoint, positive and has inverse compact on a complex Hilbert space $H=L^2(\Omega)$ $A:D(A)\subset L^2(\Omega)\con L^2(\Omega)$, where
\beqaa
A:=-\Delta,\quad D(A)=H^2(\Omega)\cap H^1_0(\Omega).\label{op-A}
\eeqaa
 Therefore, the operator $A^{\theta}$ is self-adjoint positive for all $\theta\in\R$, bounded for $\theta\leq 0$, and  the embedding
\beqae
D(A^{\theta_1})\hookrightarrow D(A^{\theta_2}),
\eeqae
is continuous for $\theta_1>\theta_2$. Here, the norm in $D(A^{\theta})$ is given by $\|u\|_{D(A^{\theta})}:=\|A^{\theta}u\|$, $u\in D(A^{\theta})$, where $\|\cdot\|$ denotes the norm in the Hilbert space $L^2(\Omega)$.  Some of these spaces are  $D(A^0)=L^2(\Omega)$, $D(A^{1/2})=H_0^1(\Omega)$ and $D(A^{-1/2})=H^{-1}(\Omega)$.

It is known that the above operator is self-adjoint, positive and has inverse compact. Using this notation to find a solution of the the system \eqref{ISplacas0}-\eqref{ISplacas2} is equivalent to find $u,v$ in some subset of $D(A)$ such that satisfy the equations
\begin{eqnarray}
\label{ISplacas-20}
    u_{tt}+\gamma A^\tau u_{tt}+ A^2 u-\alpha A\theta=0, \\
\label{ISplacas-25}
    \theta_t+\kappa A\theta+\beta A u_t =0,
\end{eqnarray}
satisfying the  boundary conditions (Hinged plate/Dirichlet temperature):
\begin{equation}\label{ISplacas-30}
    u=\Delta u=0,\quad \theta=0,\quad x\in\partial\Omega,\ t>0,
\end{equation}
and the initial data
\begin{eqnarray}
\label{ISplacas-35} u(x,0)=u_0,\ u_t(x,0)=u_1,\
\theta(x,0)=\theta_0, \qquad x\in\Omega.
\end{eqnarray}

For  $\gamma$ positive we can extend the operator $I+\gamma A^\tau$ in the following sense:
\begin{equation*}\label{BIsometrica}
(I+\gamma A^\tau)\colon D(A^{\tau/2})\to \;D(A^{-\tau/2})
\end{equation*}
defined by
\begin{equation}\label{EqPiDual}
\dual{( I+\gamma A^\tau)z_1}{z_2}_{D(A^{-\tau/2})\times D(A^{\tau/2})}=\dual{z_1}{z_2}+\gamma\dual{A^{\tau/2}z_1}{A^{\tau/2}z_2},
\end{equation}
for $z_1,z_2\in D(A^{\tau /2})$, where $\dual{\cdot}{\cdot}$ denotes the inner product in the Hilbert space $D(A^0)$.  Note that this operator is an isometric operator when we consider the equivalent norm in the space $D(A^{\tau/2})$: $\pbrack{\|z\|^2+\gamma\|A^{\tau/2} z\|^2}^{1/2}:=\|z\|^{1/2}_{D(A^\frac{\tau}{2})}.$

Now, we will use a semigroup approach to study the well-posedness of the system \eqref{ISplacas-20}-\eqref{ISplacas-35}.  Tomando $v=u_t$ and applying the operators $( I+\gamma A)^{-1}$,  to equations \eqref{ISplacas-20} and considering $U=(u,v,\theta)$ and $U_0=(u_0,v_0,\theta_0)$, the system \eqref{ISplacas0}-\eqref{ISplacas2}, can be written in the following abstract framework
\begin{equation}\label{Fabstrata}
    \frac{d}{dt}U(t)=\mathbb{B} U(t),\quad    U(0)=U_0,
\end{equation}
 where the operator $\mathbb{B}$ is given by
\begin{equation}\label{operadorAgamma}
 \mathbb{B}U:=\Big( v,(I+\gamma A^\tau)^{-1}\key{- A^2 u+\alpha A \theta}, -\kappa A\theta -\beta Av  \Big),
\end{equation}
for $U=(u,v,\theta)$. This operator will be defined in a suitable subspace of the phase space
$$
\mathbb{H}:=D(A)\times D(A^\frac{\tau}{2})\times D(A^0),
$$
which in view of (\ref{EqPiDual}), it is a Hilbert space with the inner product
\begin{eqnarray*}\nonumber
\dual{ U_1}{U_2} & := &\beta \dual{Au_1}{Au_2}_{D(A^0)\times D(A^0)}+\beta \dual{( I+\gamma A^\tau)v_1}{v_2}_{D(A^{-\frac{\tau}{2}})\times D(A^\frac{\tau}{2})}\\
& & \times \alpha \dual{\theta_1}{\theta_2}_{D(A^0)\times D(A^0)},  \label{pinternoa2}
\end{eqnarray*}
for $U_i=(u_i, v_i, \theta_i)\in \mathbb{H}$,  $i=1,2$  and induced norm
\begin{equation}\label{NORM}
\|U\|_\mathbb{H}^2:=\beta\|Au\|^2+\beta\|v\|^2_{D(A^\frac{\tau}{2})}+\alpha\|\theta\|^2.
\end{equation}
In these conditions, we define the domain of $\mathbb{B}$ as
\begin{equation}\label{dominioB}
    \mathcal{D}(\mathbb{B}):= \Big \{ U\in \mathbb{H} \colon  v\in D(A),  -A u+\alpha\theta\in D( A^{1-\frac{\tau}{2}}),  -\kappa\theta-\beta v\in D(A) \Big\}.
\end{equation}

To show that the operator $\mathbb{B}$ is the generator of a $C_0$- semigroup we invoke a result from Liu-Zheng' book Theorem\eqref{TLiuZ}.

Let us see that the operator $\mathbb{B}$ satisfies the conditions of  Theorem\eqref{TLiuZ}.  Clearly,  we see that $D(\mathbb{B})$ is dense in $\mathbb{H}$. Taking the inner product of $\mathbb{B}U$ with $U$ we have
\begin{equation}\label{eqdissipative}
\text{Re}\dual{\mathbb{B}U}{U}=  -\kappa\alpha\|A^\frac{1}{2} \theta \|^2, \quad\forall\ U\in D(\mathbb{B}),
\end{equation}
that is, the operator $\mathbb{B}$ is dissipative.\\
To complete the conditions of the above theorem, it remains to show that $0\in\rho(\mathbb{B})$. Let $F=(f,g,h)\in \mathbb{H}$, let us see that the stationary problem $ \mathbb{B}U=F$ has a solution $U=(u,v,\theta)$.  From the definition of the operator  $\mathbb{B}$ given in
\eqref{operadorAgamma} this system can be written as                     
\begin{align}\label{exist-10A}
	v=f,\qquad& \quad\quad  - A^2u+\alpha A\theta=(I+\gamma A^\tau)g\quad\quad 
	-\kappa A\theta-\beta Av=h.
\end{align}
Therefore, it is not difficult to see that there exist only one solution $u$ and $\theta$
of the system
\begin{multline}\label{Eliptico001}
 -A^2u+\alpha A\theta=(I+\gamma A^\tau)g \in D(A^\frac{\tau}{2})\\
  -\kappa A\theta=h+\beta Af \in D(A^0), 
\end{multline}
from where we have that
$$\|U\|_\mathbb{H}\leq C\|F\|_\mathbb{H},$$
wich in particular implies that $\|\mathbb{B}^{-1}F\|_\mathbb{H}\leq \|F\|_\mathbb{H}$, so we have  that   $0$ belongs to the resolvent set $\rho(\mathbb{B})$.  Consequently, from Theorem \ref{TLiuZ}  we have  $\mathbb{B}$ is the generator of a contractions semigroup.
 As a consequence of the previous Theorem\eqref{TLiuZ},  we obtain
\begin{theorem}
Given $U_0\in\mathbb{H}$ there exists a unique weak solution $U$ to  the problem \eqref{Fabstrata} satisfying 
$$U\in C([0, +\infty), \mathbb{H}).$$
Futhermore,  if $U_0\in  D(\mathcal{B}^k), \; k\in\mathbb{N}$, then the solution $U$ of \eqref{Fabstrata} satisfies
$$U\in \bigcap_{j=0}^kC^{k-j}([0,+\infty),  D(\mathcal{B}^j).$$
\end{theorem}

In what follows, $C$ and $C_\delta$  will denote a positive constant that assumes different values in different places.
\section{Regularity they Euler-Bernoulli to Kirchhoff-Love Thermoelastic Plates }

In this section we discuss the regularity of the semigroup $S(t)=e^{\mathbb{B}t}$,  in two subsections:  First we analyze the lack of analyticity of $S(t)$ for $0<\tau\leq 1$, then we study the Gevrey class of $S(t)$ for  $\frac{1}{2}\leq\tau<1$.

The following theorem characterizes the analyticity of $S(t)$:

\begin{theorem}[see \cite{LiuZ}]\label{LiuZAnaliticity}
    Let $S(t)=e^{\mathbb{B}t}$ be $C_0$-semigroup of contractions  on a Hilbert space $\mathbb{H}$.   Suppose that
    \begin{equation*}
    \rho(\mathbb{B})\supseteq\{ i\lambda/ \lambda\in \R \}  \equiv i\R
    \end{equation*}
     Then $S(t)$ is analytic if and only if
    \begin{equation}\label{Analyticity}
     \limsup\limits_{|\lambda|\to
        \infty}
    \|\lambda(i\lambda I-\mathbb{B})^{-1}\|_{\mathcal{L}(\mathbb{H})}<\infty.
    \end{equation}
    holds.
\end{theorem}

Note that to show the condition \eqref{Analyticity} it is enough to show that: Let $\delta>0$. There exists a constant  $C_\delta > 0$ such that the solutions of the system \eqref{ISplacas-20}-\eqref{ISplacas-35}  for $|\lambda|>\delta$  satisfy the inequality
\begin{equation}\label{EquivAnaliticity}
|\lambda|\|U\|_\mathbb{H}^2\leq C_\delta\|F\|_\mathbb{H}\|U\|_\mathbb{H}\qquad \Longleftrightarrow\qquad|\lambda|\|U\|_\mathbb{H}\leq C_\delta\|F\|_\mathbb{H} .
\end{equation}

 First, note that if $\lambda\in\rho(\mathbb{B})\subset \R$ and $F=(f, g, h)\in \mathbb{H}$ then the solution $U=(u,w,\theta)\in\hbox{D}(\mathbb{B
 })$ of the stationary system $(i\lambda I- \mathbb{B
 })U=F$ can be written by
\beqaa
i\lambda u-v &=& f\quad {\rm in}\quad D(A)\label{Pesp-10}\\
i\lambda( I+\gamma A^\tau) v+ A^2 u-\alpha A\theta &=&(I+\gamma A^\tau) g\quad {\rm in}\quad D(A^{-\frac{\tau}{2}})\label{Pesp-20}\\
i\lambda\theta+\kappa A\theta+\beta Av&=& h \quad{\rm in}\quad D(A^0)\label{Pesp-30}
\eeqaa
we have to
\begin{equation}\label{Pdis-10}
\kappa\alpha\|A^\frac{1}{2}\theta\|^2=\text{Re}\dual{(i\lambda -\mathbb{B})U}{U}=\text{Re}\dual{F}{U}\leq \|F\|_\mathbb{H}\|U\|_\mathbb{H}.
\end{equation}

Next we will show two lemmas that will be fundamental to achieve our results.
\begin{lemma}\label{Lemma001Exp}
Let $\delta > 0$. There exists $C_\delta > 0$ such that the solutions of the system \eqref{ISplacas-20}-\eqref{ISplacas-35} for $|\lambda|>0$,  satisfy
\begin{equation}\label{Exp000}
    \limsup\limits_{|\lambda|\to
   \infty}
   \|(i\lambda I-\mathbb{B})^{-1}\|_{\mathcal{L}(\mathbb{H})}<\infty\qquad{\rm for}\qquad 0\leq\tau\leq 1.
\end{equation}
\end{lemma}
\begin{proof}
To show the \eqref{Exp000},  it suffices to show that,  given $\delta>0$  there exists a constant $C_\delta > 0$ such that the solutions of the system \eqref{ISplacas-20}-\eqref{ISplacas-35} for $|\lambda|>\delta$  satisfy the inequality
\begin{equation}\label{EquivExp}
\|(i\lambda I-\mathbb{B})^{-1}F\|_\mathbb{H}^2=\|U\|^2_\mathbb{H}=\beta\|Au\|^2+\beta\|v\|_{D(A^\frac{\tau}{2})}^2+\alpha\|\theta\|^2 \leq C_\delta\|F\|_\mathbb{H}\|U\|_\mathbb{H}.
\end{equation}
As $0\leq\frac{1}{2}$, applying continuous immersions and the estimate \eqref{Pdis-10} we will have $\alpha\|\theta\|^2\leq C_\delta\|F\|_\mathbb {H}\|U\|_\mathbb{H}$,  therefore it remains to show that
 $$\beta[\|Au\|^2+\|v\|^2_{D(A^\frac{\tau}{2})}]\leq C_\delta\|F\|_\mathbb{H}\|U\|_\mathbb{H}. $$
Taking the duality product between  equation\eqref{Pesp-20} and $\beta u$ and using the equation \eqref{Pesp-10},  taking advantage 
of the self-adjointness of the powers of the operator $A$, we obtain
\begin{multline}\label{Exp001}
\beta\|Au\|^2= \beta\|v\|^2_{D(A^\frac{\tau}{2})}+\beta\gamma\dual{A^\frac{\tau}{2}v}{A^\frac{\tau}{2}f}\\
+\beta\dual{v}{f}+\alpha\beta\dual{\theta}{Au}+\beta\dual{g}{u}+\beta\gamma\dual{A^\frac{\tau}{2} g}{A^\frac{\tau}{2} u}.
\end{multline}
On the other hand, taking the duality product between  equation\eqref{Pesp-30} and $A^{-1}(I+\gamma A^\tau)v$,  using \eqref{Pesp-20} and  taking advantage 
of the self-adjointness of the powers of the operator $A$, we obtain
\begin{eqnarray*}
\beta\|v\|^2_{D(A^\frac{\tau}{2})} &= &\dual{A^{-1}\theta}{i\lambda(I+\gamma A^\tau)v}-\kappa\dual{\theta}{(I+\gamma A^\tau)v}+\dual{A^{-1}h}{v}+\gamma\dual{h}{A^{\tau-1}v}\\
&=&-\dual{\theta}{Au}+\alpha\|\theta\|^2+\dual{A^{-1}\theta}{g}+\gamma\dual{\theta}{A^{\tau-1}g}\\
& &-\kappa\dual{A^\frac{1}{2}\theta}{A^{-\frac{1}{2}}(I+\gamma A^\tau)v}+\dual{A^{-1}h}{v}+\gamma\dual{h}{A^{\tau-1}v}
\end{eqnarray*}
Applying Cauchy-Schwarz and Young's inequalities,   for $\varepsilon>0$ exists $C_\varepsilon>0$ such that
\begin{multline}\label{Exp002}
\beta\|v\|^2_{D(A^\frac{\tau}{2})} \leq  C_\varepsilon\|\theta\|^2+\varepsilon\|Au\|^2+C_\delta\|A^{-1}\theta\|\|g\|+\gamma\|\theta\|\|A^{\tau-1}g\|+C_\varepsilon\|A^\frac{1}{2}\theta\|^2\\+\varepsilon[\|A^{-\frac{1}{2}}v\|^2+\|A^{ \tau-\frac{1}{2}}v\|^2]+
\|A^{-1}h\|\|v\|+\gamma\|h\|\|A^{\tau-1}v\|,
\end{multline}
then, as $-\frac{1}{2}\leq \tau-\frac{1}{2}\leq\frac{\tau}{2}$,  the continuous
embedding $D(A^{\theta_2})\hookrightarrow D(A^{\theta_1})$, $\theta_2>\theta_1$, from norms $\|F\|_\mathbb{H}$ and $\|U\|_\mathbb{H}$ and estimative \eqref{Pdis-10},  we have
\begin{equation}\label{Exp003}
\beta\|v\|^2_{D(A^\frac{\tau}{2})} \leq \varepsilon\|Au\|^2+C_\delta\|F\|_\mathbb{H}\|U\|_\mathbb{H}.
\end{equation}
Using \eqref{Exp003} in \eqref{Exp001},    applying Cauchy-Schwarz and Young's inequalities  and  from norms $\|F\|_\mathbb{H}$ and $\|U\|_\mathbb{H}$,   we obtain
\begin{equation}\label{Exp004}
\beta\|Au\|^2\leq C_\delta\|F\|_\mathbb{H}\|U\|_\mathbb{H}.
\end{equation}
Now,  using  \eqref{Exp004} in \eqref{Exp003}, we have
\begin{equation}\label{Exp005}
\beta\|v\|^2_{D(A^\frac{\tau}{2})} \leq C_\delta\|F\|_\mathbb{H}\|U\|_\mathbb{H}.
\end{equation}
Therefore, from the estimates \eqref{Exp004}, \eqref{Exp005} and \eqref{Pdis-10},  we conclude the proof of \eqref{EquivExp},  thus finishing the proof of this lemma.

\end{proof}

\begin{lemma}\label{Lemma002Exp}
Let $\delta > 0$. There exists $C_\delta > 0$ such that the solutions of the system \eqref{ISplacas-20}-\eqref{ISplacas-35}  for $|\lambda|>\delta$,  satisfy
\begin{equation}\label{Exp006}
|\lambda|[\beta\|Au\|^2+\alpha\|\theta\|^2] \leq \beta|\lambda|\|v\|^2_{D(A^\frac{\tau}{2})}+C_\delta\|F\|_\mathbb{H}\|U\|_\mathbb{H}.
\end{equation}
\end{lemma}
\begin{proof}
{\bf Estimative of $\beta|\lambda|\|Au\|^2$}

Taking the duality product between  equation\eqref{Pesp-20} and $\beta\lambda u$ and using the equation \eqref{Pesp-10},  taking advantage 
of the self-adjointness of the powers of the operator $A$, we obtain
\begin{eqnarray}
\nonumber
\beta\lambda\|Au\|^2 & =& \beta\lambda\|v\|^2_{D(A^\frac{\tau}{2})}+\beta\dual{\lambda(I+\gamma A^\tau)v}{f}
+\alpha\beta\dual{\theta}{A(-iv-if)}\\
\nonumber
& & +\beta\dual{(I+\gamma A^\tau)g}{-if-iv}\\
\nonumber
&= & \beta\lambda\|v\|^2_{D(A^\frac{\tau}{2})}+i\beta\dual{Au}{Af}-i\beta\dual{\theta}{f}-i\beta\dual{(I+\gamma A^\tau)g}{f}\\
\nonumber
& & +i\alpha\beta\dual{\theta}{Av} +i\alpha\beta\dual{\theta}{f}
 +i\beta\dual{(I+\gamma A^\tau)g}{f}+i\beta\dual{(I+\gamma A^\tau)g}{v}\\\label{Exp007}
& = & \beta\lambda\|v\|^2_{D(A^\frac{\tau}{2})}+i\beta\dual{Au}{Af}-i\beta\dual{\theta}{f}+i\alpha\beta\dual{\theta}{Av}+i\alpha\beta\dual{\theta}{f}\\
\nonumber
& & +i\beta\dual{(I+\gamma A^\tau)g}{v}.
\end{eqnarray}
On the other hand,  taking the duality product between  equation\eqref{Pesp-30} and $\frac{\alpha}{\kappa}A^{-1}(\lambda\theta)$ and using the equation \eqref{Pesp-20},  taking advantage 
of the self-adjointness of the powers of the operator $A$, we obtain
\begin{eqnarray}
\label{Exp007A}
\alpha\lambda\|\theta\|^2=-i\dfrac{\alpha}{\kappa}\lambda^2\|A^{-\frac{1}{2}}\theta\|^2-\dfrac{\alpha\beta}{\kappa}\dual{v}{\lambda\theta}+\dfrac{\alpha}{\kappa}\dual{h}{A^{-1}(\lambda\theta)},
\end{eqnarray}
from:
\begin{eqnarray}
\label{Exp007B}
-\dfrac{\alpha\beta}{\kappa}\dual{v}{\lambda\theta} &= &i\alpha\beta\dual{Av}{\theta}+i\dfrac{\alpha\beta^2}{\kappa}\|A^\frac{1}{2}v\|^2+i\dfrac{\alpha\beta}{\kappa}\dual{v}{h}\\
\label{Exp007C}
\dfrac{\alpha}{\kappa}\dual{h}{A^{-1}(\lambda\theta} & = & -i\alpha\dual{h}{\theta}-i\dfrac{\alpha\beta}{\kappa}\dual{h}{v}+i\dfrac{\alpha^2}{\kappa}\|A^{-\frac{1}{2}}h\|^2,
\end{eqnarray}
Adding \eqref{Exp007} with \eqref{Exp007A} and in the result using the identities \eqref{Exp007B} and \eqref{Exp007C}, we get
\begin{multline}\label{Exp007D}
\lambda[\beta\|Au\|^2+\alpha\|\theta\|^2] =\beta\lambda\|v\|^2_{D(A^\frac{\tau}{2})}+i\beta\dual{Au}{Af}\\
-i\beta\dual{\theta}{f}+i\alpha\beta\dual{\theta}{Av}+i\alpha\beta\dual{\theta}{f}
 +i\beta\dual{(I+\gamma A^\tau)g}{v}\\-i\dfrac{\alpha}{\kappa}\lambda^2\|A^{-\frac{1}{2}}\theta\|^2
+i\alpha\beta\dual{Av}{\theta}+i\dfrac{\alpha\beta^2}{\kappa}\|A^\frac{1}{2}v\|^2+i\dfrac{\alpha\beta}{\kappa}\dual{v}{h}\\
-i\alpha\dual{h}{\theta}-i\dfrac{\alpha\beta}{\kappa}\dual{h}{v}+i\dfrac{\alpha^2}{\kappa}\|A^{-\frac{1}{2}}h\|^2.
\end{multline}
Of identity  $i\alpha\beta[\dual{\theta}{Av}+\dual{Av}{\theta}]=i2\alpha\beta{\rm Re}\dual{\theta}{Av}$, taking real part of \eqref{Exp007D},  we have
\begin{multline}\label{Exp007E}
\lambda[\beta\|Au\|^2+\alpha\|\theta\|^2] =\beta\lambda\|v\|^2_{D(A^\frac{\tau}{2})}+i\beta\dual{Au}{Af}\\
-i\beta\dual{\theta}{f}+i\alpha\beta\dual{\theta}{f}
 +i\beta\dual{(I+\gamma A^\tau)g}{v}\\
+i\dfrac{\alpha\beta}{\kappa}\dual{v}{h}
-i\alpha\dual{h}{\theta}-i\dfrac{\alpha\beta}{\kappa}\dual{h}{v}.
\end{multline}
Applying Cauchy-Schwarz and Young inequalities and norms $\|F\|_\mathbb{H}$ and $\|U\|_\mathbb{H}$,  we finish proving this lemma.

\end{proof}
\begin{lemma}\label{Lemma8}
Let $\delta>0$.  There exists $C_\delta>0$ such that the solutions of the system \eqref{ISplacas-20}-\eqref{ISplacas-35}  for $|\lambda| > \delta$,  satisfy
\begin{equation}
\|A^\frac{1}{2}v\|^2\leq C_\delta\|F\|_\mathbb{H}\|U\|_\mathbb{H}\qquad{\rm for}\qquad  \dfrac{1}{2}\leq\tau\leq 1.
\end{equation}
\end{lemma}
\begin{proof}
Performing the product of duality between equation \eqref{Pesp-30} and $A^{-\tau}(I + \gamma A^\tau )v$ and again using the property that for  all $\eta\in\mathbb{R} $, $A^\eta$ is self-adjoint, we have
\begin{eqnarray}
\nonumber
\beta\dual{Av}{A^{-\tau}(I+\gamma A^\tau)v} &= & \dual{A^{-\tau}\theta}{i\lambda (I+\gamma A^\tau)v}-\kappa\dual{A^\frac{1}{2}\theta}{A^{\frac{1}{2}-\tau}(I+\gamma A^\tau)v}\\
\label{Exp008A}
& & +\dual{h}{A^{-\tau}(I+\gamma A^\tau)v},
\end{eqnarray}
then, using \eqref{Pesp-20}  in \eqref{Exp008A},  we have
\begin{eqnarray}
\nonumber
\beta\|A^\frac{1-\tau}{2}v\|^2+\beta\gamma\|A^\frac{1}{2}v\|^2\hspace*{-0.3cm} &= &\hspace*{-0.3cm} 
\dual{A^{-\tau}\theta}{-A^2u+\alpha A\theta+(I+\gamma A^\tau)g}
-\kappa\dual{A^\frac{1}{2}\theta}{A^{\frac{1}{2}-\tau}v}\\
\label{Exp008B}
& & -\kappa\gamma\dual{A^\frac{1}{2}\theta}{A^\frac{1}{2}v}+\dual{h}{A^{-\tau}v}+\gamma\dual{h}{v}.
\end{eqnarray}
Considering continuous immersion
$D(A^{\theta_2})\hookrightarrow  D(A^{\theta_1}), \;\theta_2 >\theta_1$ 
for   $\frac{1-\tau}{2}\leq\frac{1}{2}$  and Cauchy-Schwarz inequality, we have
\begin{eqnarray}
\nonumber
\hspace*{-0.5cm}\beta\gamma\|A^\frac{1}{2}v\|^2 &\leq& |-\dual{A^\frac{1}{2}\theta}{A^{\frac{3}{2}-\tau}u}|+\alpha\|A^\frac{1-\tau}{2}\theta\|^2+|\dual{\theta}{A^{-\tau}g}|+|\gamma\dual{\theta}{g}|\\
\label{Exp008C}
& & |\kappa\dual{A^\frac{1}{2}\theta}{A^{\frac{1}{2}-\tau}v}|+|-\kappa\dual{A^\frac{1}{2}\theta}{A^\frac{1}{2}v}|+|\dual{h}{A^{-\tau}v}|+\gamma|\dual{h}{v}|.
\end{eqnarray}
Considering $\frac{3}{2}-\tau\leq 1\Leftrightarrow \frac{1}{2}\leq \tau\leq 1$, for $\varepsilon>0$ exists $C_\varepsilon>0$ independent de $\lambda$, and applying Young inequality, we have
\begin{multline}
\beta\gamma\|A^\frac{1}{2}v\|^2\leq C\{ \|A^\frac{1}{2}\theta\|^2+\|AU\|^2+\|\theta\|\|g\|+\|h\|\|A^{-\tau}v\|+\|h\|\|v\|\}\\
+\varepsilon\|A^\frac{1}{2}v\|^2+C_\varepsilon\|A^\frac{1}{2}\theta\|^2
\end{multline}
Finally using estimative \eqref{Pdis-10},  norms $\|F\|_\mathbb{H}$ and $\|U\|_\mathbb{H}$   and  applying Lemma\eqref{Lemma001Exp}. The proof of this lemma is finished.

\end{proof}
\begin{lemma}\label{EImaginary}
\label{iR}Let $\varrho(\mathbb{B})$ be the resolvent set of operator
$\mathbb{B}$. Then
\begin{equation}
i\hspace{0.5pt}\mathbb{R}\subset\varrho(\mathbb{B}).
\end{equation}
\end{lemma}
\begin{proof}
 Let us prove that $i\R\subset\rho(\mathbb{B})$ by using an argument by contradiction, so we suppose that $i\R\not\subset \rho(\mathbb{B})$. 
 As $0\in\rho(\mathbb{B})$ and $\rho(\mathbb{B})$ is open, we consider the highest positive number $\lambda_0$ such that the $]-i\lambda_0,i\lambda_0[\subset\rho(\mathbb{B})$ then $i\lambda_0$ or $-i\lambda_0$ is an element of the spectrum $\sigma(\mathbb{B})$.   
 We Suppose $i\lambda_0\in \sigma(\mathbb{B})$ (if $-i\lambda_0\in \sigma(\mathbb{B})$ the proceeding is similar). Then, for $0<\delta<\lambda_0$ there exist a sequence of real numbers $(\lambda_n)$, with $\delta\leq\lambda_n<\lambda_0$, $\lambda_n\con \lambda_0$, and a vector sequence  $U_n=(u_n,v_n,\theta_n)\in D(\mathbb{B})$ with  unitary norms, such that
\beqae
\|(i\lambda_nI-\mathbb{B}) U_n\|_\mathbb{H}=\|F_n\|_\mathbb{H}\con 0,
\eeqae
as $n\con \infty$.   From  estimative \eqref{EquivExp},  we have 
\begin{equation*}
\|U_n\|^2_\mathbb{H}=\beta\|Au_n\|^2 +\beta\|v_n\|^2_{D(A^\frac{\tau}{2})}+\alpha\|\theta_n\|^2\leq C_\delta\|F_n\|_\mathbb{H}\|U_n\|_\mathbb{H}\con 0.
\end{equation*}
Therefore, we have  $\|U_n\|_\mathbb{H}\con 0$ but this is absurd, since $\|U_n\|_\mathbb{H}=1$ for all $n\in\N$. Thus, $i\R\subset \rho(\mathbb{B})$. This completes the proof this lemma. 

\end{proof}
\subsection{Lack of analyticity of $S(t)=e^{\mathbb{B}t}$}
  
Since the operator $A$ defined on \eqref{op-A} is positive,  self-adjoint and it has compact resolvent, its spectrum is constituted by positive eigenvalues $(\sigma_n)$ such that $\sigma_n\con \infty$ as $n\con \infty$. For $n\in \N$ we denote with $e_n$ an unitary $D(A^\frac{\tau}{2})$-norm eigenvector associated to the eigenvalue $\sigma_n$, that is,
\begin{equation}\label{auto-10}
Ae_n=\sigma_ne_n,\quad \|e_n\|_{D(A^\frac{\tau}{2})}=1,\quad n\in\N.
\end{equation}

\begin{theorem}\label{AnaliticidadePlaca-20}
The semigroup $S(t)=e^{\mathbb{B}t}$  is not analytic for $\tau\in (0,1]$.
\end{theorem}
 \begin{proof}
We apply Theorem \eqref{LiuZAnaliticity}   to show this result.  Consider the eigenvalues and eigenvectors of the operator $A$ as in \eqref{auto-10}.  Let $F_n=(0,-e_n,0)\in \mathbb{H}$. The solution $U_n=(u_n,v_n,\theta_n)$ of the system $(i\lambda_n I-\mathbb{B})U_n=F_n$ satisfies $v_n=i\lambda_n u_n$ and the following equations
\begin{eqnarray*}
 \lambda^2_n (I+\gamma A^\tau) u_n-A^2u_n+\alpha A\theta_n&=& (I+\gamma A^\tau)e_n,\\
 i\lambda\theta_n +\kappa A\theta_n+i\lambda_n\beta Au_n& = & 0.
\end{eqnarray*}
 Let us see whether this system admits solutions of the form
 \begin{equation*}
    u_n=\mu \tilde{e}_n,\quad \theta_n=\nu \tilde{e}_n,
 \end{equation*}
for some complex numbers $\mu_n$ and $\nu_n$. Then, the numbers $\mu_n$, $\nu_n$ should satisfy the algebraic system
\begin{eqnarray}\label{eq01systemotimal}
 \{\lambda^2_n (1+\gamma \sigma_n^\tau)- \sigma_n^2\}\mu_n+\alpha\sigma_n\nu_n&=& (1+\gamma \sigma_n^\tau),\\
 \label{eq02systemotimal}
 i\lambda\beta\sigma_n\mu_n+(i\lambda+\kappa\sigma_n)\nu_n& = & 0.
\end{eqnarray}
In this point, we introduce the numbers
\begin{equation*}
\lambda_n^2:=\dfrac{\sigma_n^2}{1+\gamma\sigma_n^\tau}.
\end{equation*}
Thus, if we introduce the notation $x_n\approx y_n$ meaning that $\displaystyle\lim_{n\con\infty}\frac{|x_n|}{|y_n|}$ is a positive real number, we have that
\beqae
|\lambda_n|\approx |\sigma_n|^\frac{2-\tau}{2}.
\eeqae
With these considerations we have that to solving this system \eqref{eq01systemotimal}--\eqref{eq02systemotimal}, we find that
\begin{equation*}
|\mu_n|=\Big| \dfrac{\lambda_n+\gamma\lambda_n\sigma_n^\tau+i\kappa[\gamma\sigma_n^{1+\tau}+\sigma_n]}{\alpha\beta\lambda_n\sigma_n^2} \Big|\approx|\sigma_n|^\frac{3\tau-4}{2}\approx|\lambda_n|^\frac{3\tau-4}{2-\tau}.
\end{equation*}
 Therefore,  the solution $U_n$ of the system $(i\lambda_n-\mathbb{B})U_n=F_n$ satisfy 
 $$\|U_n\|_\mathbb{H}\geq K\|v_n\|_{D(A^\frac{\tau}{2})}=K|\lambda_n|\|u_n\|_{D(A^\frac{\tau}{2})}=K|\lambda_n| |\mu_n|\|e_n\|_{D(A^\frac{\tau}{2})}=K| \lambda_n|^\frac{2\tau-2}{2-\tau}.$$
     Then 
 \begin{equation*}
|\lambda_n| \|U_n\|_\mathcal{H}\geq K|\lambda|^\frac{\tau}{2-\tau}.
 \end{equation*}
Therefore $\|\lambda_n|\|U_n\|_\mathbb{H}\to\infty$ for $\tau>0$    approaches infinity as $|\lambda_n|\to\infty$. Therefore the \eqref{Analyticity} condition fails. consequently  for $ \tau>0$,  to semigroup $S(t)$ is not analytic,  in particular $S(t)$ is not analytic for $ \tau\in (\frac{1}{2},1]$. This completes the proof of this theorem.

\end{proof}

\subsection{Gevrey Class}

Before exposing our results, it is useful to recall the next definition and result  presented in \cite{Tebou2020} (adapted from
\cite{TaylorM}, Theorem 4, p. 153]).

\begin{definition}\label{Def1.1Tebou} Let $t_0\geq 0$ be a real number. A strongly continuous semigroup $S(t)$, defined on a Banach space $ \mathbb{H}$, is of Gevrey class $s > 1$ for $t > t_0$, if $S(t)$ is infinitely differentiable for $t > t_0$, and for every compact set $K \subset (t_0,\infty)$ and each $\mu > 0$, there exists a constant $ C = C(\mu, K) > 0$ such that
    \begin{equation}\label{DesigDef1.1}
    ||S^{(n)}(t)||_{\mathcal{L}( \mathcal{H})} \leq  C\mu ^n(n!)^s,  \text{ for all } \quad t \in K, n = 0,1,2...
    \end{equation}
\end{definition}
\begin{theorem}[\cite{TaylorM}]\label{Theorem1.2Tebon}
    Let $S(t)$  be a strongly continuous and bounded semigroup on a Hilbert space $ \mathbb{H}$. Suppose that the infinitesimal generator $\mathbb{B}$ of the semigroup $S(t)$ satisfies the following estimate, for some $0 < \phi < 1$:
    \begin{equation}\label{Eq1.5Tebon2020}
    \lim\limits_{|\lambda|\to\infty} \sup |\lambda |^\phi ||(i\lambda I-\mathbb{B})^{-1}||_{\mathcal{L}( \mathbb{H})} < \infty.
    \end{equation}
    Then $S(t)$  is of Gevrey  class  $s$   for $t>0$,  for every   $s >\dfrac{1}{\phi}$.
    
\end{theorem}
Our main result in this subsection is as follows:
\begin{theorem} Let  $S(t)=e^{\mathbb{B}t}$  strongly continuos-semigroups of contractions on the Hilbert space $ \mathbb{H}$, the semigroups $S(t)$ is of Grevrey class $s$ for every $s>\frac{3-\tau}{2-2\tau}$ for $\tau\in [\dfrac{1}{2},1)$, as there exists a positive constant $C$ such that we have the resolvent estimative:
    \begin{equation}\label{Eq1.6Tebon2020}
    |\lambda |^\frac{2-2\tau}{3-\tau} ||(i\lambda I-\mathbb{B})^{-1}||_{\mathcal{L}( \mathbb{H})} \leq C, \quad \lambda\in\rho(\mathbb{B})\subset \R.
    \end{equation}
\end{theorem}
\begin{proof}
Note that the estimate 
\begin{equation}\label{EqEquiv1.6Tebou}
|\lambda|^\frac{2-2\tau}{3-\tau} ||(i\lambda I-\mathbb{B})^{-1}||_\mathbb{H}=|\lambda|^\frac{2-2\tau}{3-\tau}\|U\|_\mathbb{H}\leq C_\delta\|F\|_\mathbb{H}
\end{equation}
 implies the inequality \eqref{Eq1.6Tebon2020}.  Therefore from now on we will show \eqref{EqEquiv1.6Tebou},  for this purpose let us estimate the term $|\lambda|\|A^\frac{\tau}{2}v\|^2$,  we assume $\lambda\in\rho(\mathbb{B})\subset \R$ with  $|\lambda|>1$, we shall borrow some ideias from \cite{LiuR95}. Set $v=v_1+v_2$, where $v_1\in D(A)$ and $v_2\in D(A^\frac{\tau}{2})$, with
    \begin{equation}\label{Eq110AnalyR}
    i\lambda(I+\gamma A^\tau A)v_1+Av_1=(I+\gamma A^\tau)g, \hspace{0.33cm} i\lambda(I+\gamma A^\tau )v_2=-A^2u+\alpha A\theta+Av_1.
    \end{equation}
    Firstly,  applying the product duality  the first equation in \eqref{Eq110AnalyR} by $v_1$, we have
    \begin{equation}\label{Eq111AnalyR}
    i\lambda\|v_1\|^2+i\lambda\gamma\|A^\frac{\tau}{2}v_1\|^2+ \|A^\frac{1}{2} v_1\|^2=\dual{g}{v_1}+\gamma\dual{A^\frac{\tau}{2}g} {A^\frac{\tau}{2}v_1}.
    \end{equation}
    Taking first the imaginary part of \eqref{Eq111AnalyR} and in the sequence the real part and applying Cauchy-Schwarz inequality, we have
    \begin{eqnarray*}
    |\lambda|\|v_1\|^2+\gamma|\lambda|\|A^\frac{\tau}{2}v_1\|\leq |\rm{Im}\dual{g}{v_1}|+\gamma|\rm{Im}\dual{A^\frac{\tau}{2}g}{A^\frac{\tau}{2}v_1}| \leq C\|F\|_\mathbb{H}\|U\|_\mathbb{H}\\
     \text{and}\\
    \hspace*{-4cm}\|A^\frac{1}{2}v_1\|^2\leq C\|F\|_\mathbb{H}\|U\|_\mathbb{H}.
    \end{eqnarray*}
    Equivalently
    \begin{eqnarray}\label{Eq112AnalyR}
    \|v_1\|\leq C\dfrac{[ \|F\|_\mathbb{H}\|U\|_\mathbb{H}]^\frac{1}{2}}{|\lambda|^\frac{1}{2}},\quad  \|A^\frac{\tau}{2}v_1\|\leq C\dfrac{ [\|F\|_\mathbb{H}\|U\|_\mathbb{H}]^\frac{1}{2}}{|\lambda|^\frac{1}{2}} \\
    \nonumber
    \rm{and}\\
    \label{Eq112AnalyR2}
     \|A^\frac{1}{2}v_1\|\leq C[\|F\|_\mathbb{H}\|U\|_\mathbb{H}]^\frac{1}{2}.
    \end{eqnarray}
   From $A^\frac{\tau}{2}v_2=A^\frac{\tau}{2}v+A^\frac{\tau}{2}v_1$, we have
\begin{eqnarray*}
\|A^\frac{\tau}{2}v_2\|^2 &\leq & C \{\|A^\frac{\tau}{2}v\|^2+\|A^\frac{\tau}{2}v_1\|^2\} \leq C\bigg\{\dfrac{|\lambda|+1}{|\lambda|}\bigg\} \|F\|_\mathbb{H}\|U\|_\mathbb{H}.
\end{eqnarray*}
From Lemma\eqref{Lemma001Exp}   and second inequality of  \eqref{Eq112AnalyR}. Let $\delta>0$. Exists a positive constant $C_\delta$ such that for $|\lambda|\geq\delta$, we obtain
    \begin{equation}\label{Eq01EstW}
    \|A^\frac{\tau}{2}v_2\|^2\leq C_\delta  \|F\|_\mathcal{H}\|U\|_\mathcal{H}.
    \end{equation}
    From second equation of  \eqref{Eq110AnalyR}, we have
    \begin{eqnarray*}
|\lambda||(I+\gamma A^\tau)[A^{-1}v_2]|| & \leq &C\{ \|Au\|+\|\theta\|\}+\|v_1\|,
    \end{eqnarray*}
    then, we  find
    \begin{equation*}
    |\lambda|(\|A^{-1} v_2\|^2+\gamma\|A^{\frac{\tau}{2}-1}v_2\|^2)^\frac{1}{2} \leq C\{\|Au\|+\|\theta\|\}+\|v_1\|,
    \end{equation*}
    applying Cauchy-Schwarz and Young inequalities and using first inequality of \eqref{Eq112AnalyR} and estimative \eqref{Pdis-10}   for $|\lambda|^\frac{1}{2}>1$, for $0\leq\tau\leq 1$,   we obtain
    \begin{eqnarray}\nonumber
    |\lambda|\|A^{\frac{\tau}{2}-1}v_2\| & \leq & C[\|F\|_\mathbb{H}\|U\|_\mathbb{H}]^\frac{1}{2}\bigg[\dfrac{|\lambda|^\frac{1}{2}+1}{|\lambda|^\frac{1}{2}}\bigg]\\
    \label{Eq113AnalyR}
    \|A^{\frac{\tau}{2}-1}v_2\|&\leq &C|\lambda|^{-1} [\|F\|_\mathbb{H}\|U\|_\mathbb{H}]^\frac{1}{2}
    \end{eqnarray}
    On the  other hand, from $v_2=v-v_1$, we  have
    \begin{equation}\label{Eq114AnalyR}
    \|A^\frac{1}{2} v_2\|\leq \|A^\frac{1}{2} v\|+\|A^\frac{1}{2}v_1\|.
    \end{equation}

    From  Lemma\eqref{Lemma8} and \eqref{Eq112AnalyR2},  we have
 \begin{equation}\label{Eq000Gevrey}
  \|A^\frac{1}{2} v_2\|\leq C_\delta[ \|F\|_\mathbb{H}\|U\|_\mathbb{H}]^\frac{1}{2}\quad{\rm for}\quad \dfrac{1}{2}\leq\tau\leq 1.
    \end{equation}
 
 Now, by Lion's interpolations inequality ($\frac{\tau}{2}-1< \frac{\tau}{2}\leq \frac{1}{2}$)  and estimatives \eqref{Eq113AnalyR} and \eqref{Eq000Gevrey},  we derive
 \begin{eqnarray*}
 \|A^\frac{\tau}{2}v_2|| \hspace*{-0.05cm}& \leq &\hspace*{-0.05cm} C\|A^{\frac{\tau}{2}-1}v_2\|^\frac{1-\tau}{3-\tau}\|A^\frac{1}{2}v_2\|^\frac{2}{3-\tau}\quad{\rm for}\quad \dfrac{1}{2}\leq \tau \leq 1\\
 &\leq &C |\lambda|^{-\frac{1-\tau}{3-\tau}}[\|F\|_\mathbb{H}\|U\|_\mathbb{H}]^\frac{1}{2}\quad{\rm for}\quad \dfrac{1}{2}\leq\tau\leq 1.
 \end{eqnarray*}
 
 On the other hand,  as $\|A^\frac{\tau}{2}v\|\leq C\{\|A^\frac{\tau}{2}v_1\|+\|A^\frac{\tau}{2}v_1\|\}$ and $-\frac{1-\tau}{3-\tau}\geq -\frac{1}{2}$, we have
 \begin{eqnarray*}
 \|A^\frac{\tau}{2} v\| &\leq & C\{ |\lambda|^{-\frac{1-\tau}{3-\tau}}+|\lambda|^{-\frac{1}{2}}\}[\|F\|_\mathbb{H}\|U\|_\mathbb{H}]^\frac{1}{2}\\
 & \leq &  C|\lambda|^{-\frac{1-\tau}{3-\tau}}[\|F\|_\mathbb{H}\|U\|_\mathbb{H}]^\frac{1}{2}\quad{\rm for}\quad \dfrac{1}{2}\leq \tau\leq 1.
 \end{eqnarray*}
    Equivalently
 \begin{equation}\label{Eq001Gevrey}
 |\lambda|\|A^\frac{\tau}{2} v\|^2 \leq C_\delta|\lambda|^\frac{1+\tau}{3-\tau}\|F\|_\mathbb{H}\|U\|_\mathbb{H}\quad{\rm for }\quad \dfrac{1}{2}\leq\tau\leq 1.
 \end{equation}
 
 Using  \eqref{Eq001Gevrey} in  estimative  \eqref{Exp006} to Lemma\eqref{Lemma002Exp},   we have
 \begin{equation}\label{Eq002Gevrey}
 |\lambda|[\beta\|Au\|^2+\alpha \|\theta\|^2]\leq C_\delta|\lambda|^\frac{1+\tau}{3-\tau}\|F\|_\mathbb{H}\|U\|_\mathbb{H}\qquad{\rm for}\qquad \dfrac{1}{2}\leq\tau\leq 1.
 \end{equation}
 And as \eqref{Eq001Gevrey} it implies 
 \begin{equation}\label{Eq003Gevrey}
  |\lambda|\beta\| v\|^2_{D(A^\frac{\tau}{2})} \leq C_\delta|\lambda|^\frac{1+\tau}{3-\tau}\|F\|_\mathbb{H}\|U\|_\mathbb{H}\quad{\rm for }\quad \dfrac{1}{2}\leq\tau\leq 1.
 \end{equation}
Finally, using \eqref{Eq002Gevrey} and \eqref{Eq003Gevrey}.  We finish the proof of this theorem.
\end{proof}

\begin{remark}[Asymptotic Behavior] 
Note that the semigroup with class Gevrey has more regular properties than a differentiable semigroup, but is less regular than an analytic semigroup. The Gevrey rate $s>\frac{1}{\phi}$ `measures' the degree of divergence of its power series. It should be noted that the Gevrey class or analyticity of the particular model implies three important properties. The first is the property of the smoothing effect on the initial data, that is, no matter how irregular the initial data is, the model solutions are very smooth in positive time. The second property is that systems are exponentially stable. Finally, the systems enjoy the property of linear stability, which means that the type
of the semigroup is equal to the spectral limit of its infinitesimal operator. Thus, as a consequence of the lemmas \eqref{Lemma001Exp} and \eqref{EImaginary},  we have that the semigroup $S(t)=e^{\mathbb{B}t}$ is exponentially stable for all $\tau\in [ 0.1]$.
\end{remark}


\end{document}